\documentclass[notitlepage]{extarticle}

\usepackage{amssymb}
\usepackage{mathtools}
\usepackage{indentfirst}
\usepackage{graphicx}
\usepackage[colorinlistoftodos]{todonotes}
\usepackage{csquotes}
\usepackage[margin=1in]{geometry}
\usepackage{titlesec}
\usepackage[titletoc,toc,title]{appendix}
\usepackage{fancyhdr}
\usepackage{tcolorbox}
\usepackage{tabu}
\usepackage{floatrow}
\usepackage{float}
\usepackage{titling}
\usepackage{amsthm}
\newtheorem{theorem}{Theorem}[section]

\newtheorem{lemma}[theorem]{Lemma}
\newtheorem{remark}{Remark}

\title{Development of Singularities of the Skyrme Model}
\author{Michael McNulty}
\date{\today}

\begin{document}

\maketitle
\begin{abstract}
The Skyrme model is a geometric field theory and a quasilinear modification of the Nonlinear Sigma Model (Wave Maps).  In this paper we study the development of singularities for the equivariant Skyrme Model, in the strong-field limit, where the restoration of scale invariance allows us to look for self-similar blow-up behavior. After introducing the Skyrme Model and reviewing what's known about formation of singularities in equivariant Wave Maps, we prove the existence of smooth self-similar solutions to the $5+1$-dimensional Skyrme Model in the strong-field limit, and use that to conclude that the solution to the corresponding Cauchy problem blows up in finite time, starting from a particular class of everywhere smooth initial data. \end{abstract}

\section{Background}
One of the most extensively studied geometric field theories is Wave Maps. In this field theory, one studies a map from the $m+1$-dimensional Minkowski space, denoted by $\mathbb{R}^{1,m}$, with the usual Minkowski metric, denoted by $g$, to a complete $n$-dimensional Riemannian manifold $(\mathcal{N},h)$. A Wave Map, $U:\mathbb{R}^{1,m}\rightarrow\mathcal{N}$, is a critical point of the following action functional:
	\begin{equation}
		\mathcal{A}[U]=\frac{1}{2}\int tr_g\big(S(U)\big)\;d\mu_g \label{Wave Maps Action}
	\end{equation}
where $S(U):=U^*h$. The corresponding Euler-Lagrange Equation is the following nonlinear wave equation:
	\begin{equation}
		\Box_gU^a =-\Gamma^a_{bc}(U)\partial_\mu U^b\partial^\mu U^c \label{Wave Maps PDE}
	\end{equation}
where $\Gamma^a_{bc}$ are the Christoffel symbols of the metric $h$. Much is known of this equation already. Of particular interest is its development of singularities in the equivariant case with $m=3$, $\mathcal{N}=S^3$, and $h$ the standard round metric established by Shatah (see \cite{S_wksln}) and then generalized to rotationally symmetric, non-convex Riemannian manifolds by Shatah and Tahvildar-Zadeh (see \cite{SS_Cauchy} and \cite{SCS_98}).

The Skyrme Model is a quasilinear adaptation of Wave Maps, originally proposed by physicist Tony Skyrme (see \cite{Sk_unifld} and \cite{Sk_Nonlin}) for applications to particle physics. Given $(\mathbb{R}^{1,m},g)$ and $(\mathcal{N},h)$ as above, a Skyrme Map, $U:\mathbb{R}^{1,m}\rightarrow\mathcal{N}$, is a critical point of the following functional:
\begin{equation}
\mathcal{S}[U]=\int\bigg[\frac{\alpha^2}{2}tr_g\big(S(U)\big)-\frac{\beta^2}{4}\Big(tr_g\big(S^2(U)\big)-tr_g^2\big(S(U)\big)\Big)\bigg]\;\;d\mu_g \label{Skyrme Maps Action}
\end{equation}
for $\alpha,\beta\in\mathbb{R}$. In fact, the integrand of \eqref{Skyrme Maps Action} is a combination of the first two symmetric polynomials\footnote{Given an $n\times n$ matrix, $A$, with eigenvalues $\{\lambda_i\}_{i=1}^n$, we call $tr(A)=\sum_{i=1}^n\lambda_i$ the first symmetric polynomial of $A$ and $tr(A^2)-tr^2(A)=\sum_{i=1}^n\sum_{j=1,j\neq i}^n\lambda_i\lambda_j$ the second symmetric polynomial of $A$.} of $S(U)$. One can immediately see that when $\beta=0$ and $\alpha=\pm1$, we obtain \eqref{Wave Maps Action}. The corresponding Euler-Lagrange equation has been studied recently (see \cite{Geba_Lrgd} and \cite{G_2p1}). In particular, the Skyrme Model has been shown to posess large data global regularity in the equivariant case by Geba (see \cite{Geba_Lrgd}) when $n=3$. 

\section{Main Problem and Main Result}
We concern ourselves with the development of singularities of Skyrme Maps for the equivariant case of the Skyrme Model with $m=5$, $\mathcal{N}=S^5$, and $h$ the standard round metric in the strong-field limit. The solution to the equivariant, strong-field Skyrme Model equation of motion will be shown to blow up in finite time by the same mechanism as in \cite{S_wksln}.

The strong-field limit of the Skyrme Model is the limit of \eqref{Skyrme Maps Action} when $\alpha\rightarrow0$. Furthermore, an equivariant Skyrme Map, $U:\mathbb{R}^{1,5}\rightarrow S^5$, is a map of the form
\begin{equation}
U(t,r,\omega)=\big(u(t,r),\omega\big)
\end{equation}
for some unknown $u:\mathbb{R}\times\mathbb{R}_{\geq0}\rightarrow[0,\pi]$ where $t\in\mathbb{R}$ is the time coordinate, $r\in\mathbb{R}_{\geq0}$ is the radial coordinate of $\mathbb{R}^5$, and $\omega\in S^4\subset\mathbb{R}^5$. Furthermore, under the strong-field limit and equivariant ansatz, the corresponding Euler-Lagrange equation for $u$ is the semilinear wave equation
\begin{equation}
u_{tt}-u_{rr}-\frac{2}{r}u_r+\cos u\bigg(\frac{u_t^2-u_r^2}{\sin u}+\frac{3\sin u}{r^2}\bigg)=0. \label{Skyrme Maps PDE}
\end{equation}
The following theorem is the main result of this paper:

\begin{theorem}
Let $\alpha=0$ in \eqref{Skyrme Maps Action}. There exists a class of smooth initial data such that the corresponding Cauchy problem for the Euler-Lagrange equation of an equivariant Skyrme Map from $\mathbb{R}^{1,5}$ into $S^5$, in the strong field-limit, has a solution that blows up in finite time. \label{Main Result}
\end{theorem}

\section{Summary of the Proof}
Our goal is to construct smooth initial data for \eqref{Skyrme Maps PDE} which will develop a singularity in finite time. We will find such initial data by exploiting the scaling invariance of \eqref{Skyrme Maps PDE}. That is, for any $\lambda\in\mathbb{R}-\{0\}$, \eqref{Skyrme Maps PDE} is invariant under the map $(t,r)\mapsto(\lambda t,\lambda r)$. Thus, we are motivated to find a self-similar solution $u(t,r)=w(-r/t)$ for some unknown $w:\mathbb{R}\to[0,\pi]$. For convenience, we define $\rho:=\dfrac{-r}{t}$. Such a nontrivial solution is constant along rays emanating from the origin of the Minkowski space and is thus multi-valued at the origin. This forces the derivative of $u$ to become unbounded and, consequently, a singularity develops. 

Substituting $w$ into \eqref{Skyrme Maps PDE} results in the following ordinary differential equation
\begin{equation}
w_{\rho\rho}+\frac{2}{\rho}w_\rho-\bigg[\frac{3\sin^2w}{\rho^2(1-\rho^2)}-w_\rho^2\bigg]\cot w=0. \label{Self Similar Skyrme ODE}
\end{equation}
We can modify \eqref{Self Similar Skyrme ODE} by setting $w=\cos^{-1} y$ for some unknown $y$, resulting in
\begin{equation}
y_{\rho\rho}+\dfrac{2}{\rho}y_\rho+\dfrac{3y\big(1-y^2\big)}{\rho^2\big(1-\rho^2\big)}=0. \label{Reduced Self Similar Skyrme Maps ODE}
\end{equation}
If we can find a smooth solution to \eqref{Reduced Self Similar Skyrme Maps ODE} for $\rho\in[0,1]$ satisfying the regularity conditions $y(0)=\pm1$ and $y(1)=0$, then we can use that solution to specify smooth initial data in $B_1(0)\subset\Sigma_{-1}$ where $\Sigma_{-1}$ denotes the Cauchy hypersurface at $t=-1$. We can look in the past light cone of the origin of the Minkowski space in order to deduce that the derivative of the solution blows up at the origin.

First, we will show that an $H^1$ solution of \eqref{Reduced Self Similar Skyrme Maps ODE} which is both continuous in $[0,1]$ and satisfies the regularity conditions is, in fact, a smooth solution of \eqref{Reduced Self Similar Skyrme Maps ODE} . Then, we will set up a variational problem for which the critical points of some functional are solutions to \eqref{Reduced Self Similar Skyrme Maps ODE}. We will show that this functional achieves its minimum in the space for which it is defined and that this minimum has the necessary properties to be smooth as stated above. 

\section{Proof of Main Result}
\begin{remark}
We point out for notational convenience that by $B_R$, we mean $B_R(0)\subset\mathbb{R}^3$, the open ball of radius $R$ centered at the origin. Whenever we write $H^1$ or other function spaces, we always mean those functions on the ball $B_1$ unless stated otherwise. Furthermore, we will recycle the letter $C$ to denote a generic constant in our inequalities. Unless the explicit nature of this constant is important, $C$ may be a different constant in each instance.
\end{remark}
\begin{lemma} 
Let $y\in H^1$ be a solution to \eqref{Reduced Self Similar Skyrme Maps ODE} such that $y\in C[0,1]$, $y(0)=\pm1$, and $y(1)=0$. Then $y\in C^\infty[0,1]$. \label{Smoothness Lemma}
\end{lemma}
\begin{proof}
The only values of $\rho$ for which a solution of \eqref{Reduced Self Similar Skyrme Maps ODE} may not be smooth on the unit interval are $\rho=0$ and $\rho=1$. Since $y\in H^1$, then $y\in C^{0,\frac{1}{2}}(B_1\setminus B_\alpha)$ for some fixed $\alpha\in(0,1)$ by Sobolev embedding. For $\rho\in(\alpha,1]$,
\begin{align}
\begin{split}
	\big|y(1)-y(\rho)\big|&=\big|y(\rho)\big|
	\\
	&<C\big|1-\rho\big|^\frac{1}{2}. \label{Holder Continuous}
\end{split}
\end{align}
Now, define $h_y(\rho):=\dfrac{-3y(\rho)\big(1-y^2(\rho)\big)}{\rho^2(1-\rho^2)}$. Since $y\in C(\alpha,1]$, 
\begin{equation}
\big|h_y(\rho)\big|\leq C_\alpha\dfrac{|y(\rho)|}{1-\rho} \label{bound on nonlinearity}
\end{equation}
for some constant $C_\alpha$ depending on $\alpha$. Consequently, for any $p\in\mathbb{N}$,
\begin{align}
\begin{split}
	\int_\alpha^1|h_y|^p\;\rho^2d\rho&\leq C_\alpha\int_\alpha^1\frac{|y|^p}{|1-\rho|^p}d\rho
	\\
	&\leq C_\alpha\Bigg(\int_\alpha^1\bigg(\frac{|y|}{1-\rho}\bigg)^{\frac{p}{p+1}}\;d\rho\Bigg)^{p+1}\Bigg(\int_\alpha^11\;d\rho\Bigg)^{\frac{p}{p+1}}
	\\
	&\leq C_\alpha(1-\alpha)^{\frac{p}{p+1}}\Bigg(\int_\alpha^1(1-\rho)^{\frac{-p}{2(p+1)}}\;d\rho\Bigg)^{p+1}
	\\
	&<\infty.
\end{split}
\end{align}
So, $h_y\in L^p(B_1\setminus B_\alpha)$. 
Since \eqref{Reduced Self Similar Skyrme Maps ODE} can be rewritten as $\Delta y=h_y$ and $h_y\in L^p(B_1\setminus B_\alpha)$, we have that $y\in W^{2,p}(B_1\setminus B_\alpha)$. Further, $y\in C^{k,\beta}(B_1\setminus B_\alpha)$ for $k+\beta=p-\dfrac{3}{2}$ by Sobolev embedding. Thus, for any $k\in\mathbb{N}$, we can always find a $p$ which guarantees $y\in C^k(\alpha,1]$. Therefore, $y$ is a smooth function on the interval $(\alpha,1]$.

In order to show that $y$ is smooth at $\rho=0$, we change dependent variable. If $y(0)=1$, then we change to $z=y-1$. Similarly, if $y(0)=-1$, then we change to $z=y+1$. Each case is handled similarly with the appropriate change of sign. Without any loss of generality, we assume $y(0)=-1$ and change dependent variable to $z=y+1$. \eqref{Reduced Self Similar Skyrme Maps ODE} becomes:
\begin{align}
z_{\rho\rho}+\dfrac{2}{\rho}z_\rho-\dfrac{3z(z-1)\big(z-2\big)}{\rho^2\big(1-\rho^2\big)}=0 \label{Shifted Reduced Self Similar Skyrme Maps ODE}
\end{align}
with $z(0)=0$ and $z(1)=1$. Furthermore, since $y\in C[0,1]$, we also have that $z\in C[0,1]$. We will show that the nonlinearity in \eqref{Shifted Reduced Self Similar Skyrme Maps ODE} is integrable near $\rho=0$. Using this, we will show that the corresponding solution is smooth at $\rho=0$.

Multiplying \eqref{Shifted Reduced Self Similar Skyrme Maps ODE} by $z$ and integrating from some $\varepsilon$ to $\delta$, $0<\varepsilon<\delta<1$ yields
\begin{align}
\begin{split}
&\int_\varepsilon^\delta zz_{\rho\rho}+\dfrac{2}{\rho}zz_\rho-\dfrac{3z^2(z-1)\big(z-2\big)}{\rho^2\big(1-\rho^2\big)}\;\;d\rho
\\
&=\int_\varepsilon^\delta -z_{\rho}^2+\dfrac{2}{\rho}zz_\rho-\dfrac{3z^2(z-1)\big(z-2\big)}{\rho^2\big(1-\rho^2\big)}\;\;d\rho+\frac{1}{2}\partial_\rho(z^2)\big|_\varepsilon^\delta 
\\
&= 0. \label{Multiply by z and Integrate}
\end{split}
\end{align}
This implies
\begin{align}
\begin{split}
\int_\varepsilon^\delta z_{\rho}^2&-\dfrac{2}{\rho}zz_\rho+\dfrac{3z^2(z-1)\big(z-2\big)}{\rho^2\big(1-\rho^2\big)}\;\;d\rho = \frac{1}{2}\partial_\rho\big(z^2(\delta)\big)-\frac{1}{2}\partial_\rho\big(z^2(\varepsilon)\big) 
\\
&\leq \frac{1}{2}\partial_\rho\big(z^2(\delta)\big) \label{Multiply by z and Integrate to 0}
\end{split}
\end{align}
since $z^2(0)=0$. So, \eqref{Multiply by z and Integrate to 0} implies that we can take $\varepsilon\rightarrow0$. For any $a<b$ and $a>0$,
\begin{equation}
z_\rho^2-\frac{2}{\rho}zz_\rho+\dfrac{3z^2(z-1)\big(z-2\big)}{\rho^2\big(1-\rho^2\big)} = \Big(1-\frac{a^2}{b^2}\Big)z_\rho^2+\Big(\frac{a}{b}z_\rho-\frac{b}{a}\frac{z}{\rho}\Big)^2+\Bigg[\frac{3z^2(z-1)\big(z-2\big)}{\rho^2\big(1-\rho^2\big)}-\frac{b^2}{a^2}\frac{z^2}{\rho^2}\Bigg]. \label{Completing the Square}
\end{equation}
We can pick $\delta$ small enough so that $3z^2(z-1)\big(z-2\big)\approx6z^2$ since $z$ is continuous on $[0,\delta)$. The third term in \eqref{Completing the Square} becomes
\begin{align}
\begin{split}
&\frac{3z^2(z-1)\big(z-2\big)}{\rho^2\big(1-\rho^2\big)}-\frac{b^2}{a^2}\frac{z^2}{\rho^2} \approx \Big(\frac{6a^2-b^2(1-\delta^2)}{a^2(1-\delta^2)}\Big)\frac{z^2}{\rho^2}.
\end{split}
\end{align}
Define the set
\begin{align}
\mathbb{A}:=\Bigg\{(a,b)\in\mathbb{N}\times\mathbb{N}:0<a<b,\;\;\frac{6a^2-b^2(1-\delta^2)}{a^2(1-\delta^2)}>0\Bigg\}.
\end{align}
For any such $\delta$, it is possible to find a constant $C$ depending on $\delta$, such that
\begin{equation}
C(\delta)=\min_{(a,b)\in\mathbb{A}}\Bigg\{1-\frac{a^2}{b^2},\frac{6a^2-b^2(1-\delta^2)}{a^2(1-\delta^2)}\Bigg\}.
\end{equation}
Then, we can bound the first and third terms of \eqref{Completing the Square} from below by the following
\begin{align}
\begin{split}
z_\rho^2-\frac{2z_\rho  z}{\rho}+\frac{3z^2(z-1)(z-2)}{\rho^2(1-\rho^2)} &\geq \Big(1-\frac{a^2}{b^2}\Big)z_\rho^2+\Bigg[\frac{3z^2(z-1)\big(z-2\big)}{\rho^2\big(1-\rho^2\big)}-\frac{b^2}{a^2}\frac{z^2}{\rho^2}\Bigg] 
\\
&\geq C(\delta)\Big(z_\rho^2+\frac{z^2}{\rho^2}\Big). \label{Bounding Completed Square From Below}
\end{split}
\end{align}
This implies
\begin{align}
\begin{split}
&z_\rho^2+\frac{z^2}{\rho^2} \leq \frac{1}{C(\delta)}\Bigg[z_\rho^2-\frac{2z_\rho z}{\rho}+\frac{3z^2(z-1)\big(z-2\big)}{\rho^2\big(1-\rho^2\big)}\Bigg] \label{Bound for the Integral of the Nonlinearity}.
\end{split}
\end{align}
Further, \eqref{Multiply by z and Integrate to 0} and \eqref{Bound for the Integral of the Nonlinearity} imply
\begin{align}
\begin{split}
\int_0^\delta z_\rho^2+\frac{z^2}{\rho^2}\;\;d\rho&\leq\frac{1}{C(\delta)}\int_0^\delta z_\rho^2-\frac{2z_\rho z}{\rho}+\frac{3z^2(z-1)\big(z-2\big)}{\rho^2\big(1-\rho^2\big)}\;\;d\rho
\\
&\leq\frac{1}{2C(\delta)}\partial_\rho\big(z^2(\delta)\big). \label{Final Bound of the Integral of the Nonlinearity}
\end{split}
\end{align}
Thus, for any $\delta>0$ we pick $\alpha=\dfrac{1}{2}\delta$. By taking $\delta>0$, we guarantee that \eqref{Final Bound of the Integral of the Nonlinearity} is finite. This implies that
\begin{align}
\begin{split}
z_{\rho\rho}+\frac{2}{\rho}z_\rho-\frac{6z}{\rho^2}&=\frac{3z(z-1)\big(z-2\big)}{\rho^2\big(1-\rho^2\big)}-\frac{6z}{\rho^2}
\\
&=:\frac{f(z)}{\rho^2}
\\
&\leq C\frac{z^2}{\rho^2} \label{Equation is Bounded to z Squared}
\end{split}
\end{align}
is integrable on $[0,\delta)$. 

Now, we will show that $z$, the solution to \eqref{Shifted Reduced Self Similar Skyrme Maps ODE}, is smooth at $\rho=0$. Let $\rho=e^t$ for $t\in(-\infty,0]$. Then \eqref{Equation is Bounded to z Squared} becomes
\begin{equation}
z''+z'-6z=f(z) \label{Modified Shifted Reduced Self Similar Skyrme Maps ODE}
\end{equation}
where the prime now denotes derivative with respect to $t$. The solutions to the homogeneous problem are $e^{-3t}$ and $e^{2t}$. Variation of parameters tells us that the solution to \eqref{Modified Shifted Reduced Self Similar Skyrme Maps ODE} is
\begin{equation}
z(t)=\lim_{a\rightarrow-\infty}\Bigg[e^{-3t}\Big(A(a)+\int_a^tf\big(z(s)\big)e^{3s}\;\;ds\Big)-e^{2t}\Big(B(a)+\int_a^tf\big(z(s)\big)e^{-2s}\;\;ds\Big)\Bigg]. \label{Full Solution z}
\end{equation}
Introduce a new parameter $b<a$ in order to rewrite \eqref{Full Solution z} as 
\begin{equation}
z(t)=\lim_{a\rightarrow-\infty}\Bigg[e^{-3t}\Big(A(a) + \int_a^bf\big(z(s)\big)e^{3s}\;\;ds +\int_b^tf\big(z(s)\big)e^{3s}\;\;ds\Big)-e^{2t}\Big(B(a)+\int_a^tf\big(z(s)\big)e^{-2s}\;\;ds\Big)\Bigg]. \label{Full Solution z with b}
\end{equation}
We look at the limit $b\rightarrow-\infty$ and then $t\rightarrow-\infty$. First, note that $\lim_{t\rightarrow-\infty}z(t)=0$ since $z$ is assumed to be continuous. Clearly,
\begin{equation}
e^{2t}\Big(B(a)+\int_a^tf\big(z(s)\big)e^{-2s}\;\;ds\Big)\rightarrow0 \label{Second Term Vanishes}
\end{equation}
as $t\rightarrow-\infty$. Also, 
\begin{align}
\begin{split}
e^{-3t}\int_b^t\Big|f\big(z(s)\big)\Big|e^{3s}\;\;ds&\leq\int_b^t\Big|f\big(z(s)\big)\Big|\;\;ds
\\
&\rightarrow0 \label{First Second Term Vanishes}
\end{split}
\end{align}
as $b\rightarrow-\infty$ and $t\rightarrow-\infty$ since $f$ is a polynomial in $z$. Further,
\begin{equation}
 \int_a^bf\big(z(s)\big)e^{3s}\;\;ds\rightarrow0
\end{equation}
as $b\rightarrow-\infty$. So, it must be the case that
\begin{equation}
e^{-3t}A(a)\rightarrow0 \label{First First Term Vanishes}
\end{equation}
as $t\rightarrow-\infty$ by the continuity of $z$. Since $A$ is independent of $t$, $A\equiv0$. Now, examine the first derivative of $z$,
\begin{equation}
z'(t)=-3e^{-3t}\int_{-\infty}^tf\big(z(s)\big)e^{3s}\;\;ds-2e^{2t}\lim_{a\rightarrow-\infty}\Bigg(B(a)+\int_{a}^tf\big(z(s)\big)e^{-2s}\;\;ds\Bigg).
\end{equation}
As $t\rightarrow-\infty$, the second term goes to $0$ due to \eqref{Second Term Vanishes}. The first term goes to $0$ due to \eqref{First Second Term Vanishes}. So, $z'(t)\rightarrow0$ as $t\rightarrow-\infty$. This can only be the case if the solution is on the unstable manifold of \eqref{Modified Shifted Reduced Self Similar Skyrme Maps ODE}, implying $|e^{-2t}z(t)|<1$ for sufficiently small $t$. Thus, in a small neighborhood around $0$, $|z(\rho)|\leq\rho^2$ implying that $z$ is $C^1[0,\delta)$ and consequently a smooth function of $\rho$ in that neighborhood. Combining this with the result from $(\alpha,1]$, we obtain that $y$ is a smooth function of $\rho\in[0,1]$.
\end{proof}

Next, we will find a solution to \eqref{Reduced Self Similar Skyrme Maps ODE} which satisfies the hypotheses of Lemma \ref{Smoothness Lemma}. So, we will consider a variational problem with the functional
\begin{equation}
J[\psi]=\frac{1}{2}\int_0^1\Bigg[\psi_\rho^2-\frac{1}{\rho^2(1-\rho^2)}3\psi^2\Big(1-\frac{1}{2}\psi^2\Big)\Bigg]\;\;\rho^2d\rho, \label{Reduced Self Similar Skyrme Maps Minimization Functional}
\end{equation}
defined over the space
\begin{equation}
X:=\Big\{\psi\in H^1:\psi\; \text{radial},\;\psi(1)=0\Big\}. \label{space of functions}
\end{equation}
It is a routine calculation to show that critical points of \eqref{Reduced Self Similar Skyrme Maps Minimization Functional} satisfy \eqref{Reduced Self Similar Skyrme Maps ODE}. We choose to regularize $J$ by considering the functional
\begin{align}
\begin{split}
&J[\psi]=\frac{1}{2}\int_0^1\Bigg[\psi_\rho^2+\frac{1}{\rho^2(1-\rho^2)}F(\psi)\Bigg]\;\;\rho^2d\rho, \label{Regularized Reduced Self Similar Skyrme Maps Minimization Functional}
\\
&F(\psi)=
\begin{cases}
-3\psi^2\Big(1-\frac{1}{2}\psi^2\Big);  &|\psi|<1
\\
\varphi(\psi); &\text{otherwise}
\\
0; &|\psi|\geq\sqrt{2}
\end{cases}
\end{split}
\end{align}
where $\varphi(\psi)$ is a smooth function of $\rho\in[0,1]$ for any $\psi\in X$ such that $\varphi$ increases(decreases) monotonically to(from) $-3\psi^2\Big(1-\frac{1}{2}\psi^2\Big)$ from(to) $0$ for values of $\rho$ in which $1\leq|\psi(\rho)|\leq\sqrt{2}$. Along the way, we will show that our result is independent of the regularization we made.

\begin{lemma}
$J$ is a $C^1$ functional on $X$ that is bounded from below. In particular, $J$ and its first derivative are Lipschitz continuous on $X$. \label{C^1 Functional Bounded Below Lemma}
\end{lemma}
\begin{proof}
For any $u,v\in X$,
\begin{align}
\begin{split}
\frac{1}{2}\Bigg|\int_0^1\frac{3}{\rho^2(1-\rho^2)}\bigg[u^2\Big(1-\frac{1}{2}u^2\Big)-v^2\Big(1&-\frac{1}{2}v^2\Big)\bigg]\;\;\rho^2d\rho\Bigg| \leq C\int_0^1\frac{|u-v|}{\rho^2(1-\rho)}\bigg|(u+v)\Big(1-\frac{1}{2}(u^2+v^2)\Big)\bigg|\;\;\rho^2d\rho
\\
&\leq C\int_0^1\frac{|u-v|}{\rho^2(1-\rho)}\;\;\rho^2d\rho. \label{Bounding the Functional}
\end{split}
\end{align}
Integrating from $0$ to $\beta$ with $\beta\in(0,1)$,
\begin{align}
\begin{split}
\int_0^\beta\frac{|u-v|}{\rho^2(1-\rho)}\;\;\rho^2d\rho &\leq C\Bigg(\int_0^\beta|u-v|^6\;\;\rho^2d\rho\Bigg)^{1/6}\Bigg(\int_0^\beta \rho^{-5/2} d\rho\Bigg)^{5/6} 
\\
&\leq C\|u_\rho-v_\rho\|_{L^2} \label{Bounding the Functional from 0 to alpha}
\end{split}
\end{align}
and from $\beta$ to $1$
\begin{align}
\begin{split}
\int_\beta^1\frac{|u-v|}{\rho^2(1-\rho)}\;\;\rho^2d\rho &\leq C\Bigg(\int_\beta^1|u-v|^6\;\;\rho^2d\rho\Bigg)^{1/9}\Bigg(\int_\beta^1 \frac{|u-v|^{3/8}}{(1-\rho)^{9/8}}\;\;d\rho\Bigg)^{8/9}
\\
&\leq C\|u_\rho-v_\rho\|_{L^2}\Bigg(\int_\beta^1 \frac{d\rho}{(1-\rho)^{15/16}}\Bigg)^{8/9} 
\\
&\leq C\|u_\rho-v_\rho\|_{L^2}. \label{Bounding the Functional from alpha to 1}
\end{split}
\end{align}
This implies
\begin{align}
\begin{split}
\Big|J[u]-J[v]\Big|&<C\|u_\rho-v_\rho\|_{L^2}
\\
&\leq C\|u_\rho-v_\rho\|_{H^1}.\label{J is Continuous}
\end{split}
\end{align}
Thus, $J$ is Lipschitz continuous on $X$. Further,
\begin{align}
\int_0^1\Bigg|\frac{u(1-u^2)-v(1-v^2)}{\rho^2(1-\rho^2)}\Bigg|\;\;\rho^2d\rho\leq C\int_0^1\frac{|u-v|}{\rho^2(1-\rho)}\;\;\rho^2d\rho.
\end{align}
So, \eqref{Bounding the Functional from 0 to alpha} and \eqref{Bounding the Functional from alpha to 1} imply that $J$ is $C^1$ on $X$ and, more specifically, $J'$ is Lipschitz continuous on $X$.

Now, compute \eqref{J is Continuous} with $v\equiv0$. The following holds:
\begin{align}
\begin{split}
J[u]&=C\|u_\rho\|_{L^2}^2-\frac{1}{2}\int_0^1\frac{3}{\rho^2(1-\rho^2)}u^2\Big(1-\frac{1}{2}u^2\Big)\;\;\rho^2d\rho
\\
&\geq C\big(\|u_\rho\|_{L^2}^2-\|u_\rho\|_{L^2}\big). \label{Functional is Bounded Below}
\end{split}
\end{align}
Thus, $J$ is bounded from below.
\end{proof}

\begin{lemma}
If $y\in X$ is a minimizer of $J$, then $y(0)\neq0$. \label{y(0) lemma}
\end{lemma}
\begin{proof}
Since $y\in X$ is a minimizer of \eqref{Regularized Reduced Self Similar Skyrme Maps Minimization Functional}, it satisfies the Euler-Lagrange equation \eqref{Reduced Self Similar Skyrme Maps ODE}. We can convert \eqref{Reduced Self Similar Skyrme Maps ODE} to the three-dimensional, autonomous smooth dynamical system:
\begin{align}
\begin{bmatrix}
\dot{y} 
\\
\dot{q}
\\
\dot{\rho}
\end{bmatrix}
=
\begin{bmatrix}
\big(1-\rho^2\big)q
\\
\big(\rho^2-1\big)q-3y\big(1-y^2\big)
\\
\rho\big(1-\rho^2\big)
\end{bmatrix}
=:\dot{Y}(y,q,\rho)
\end{align}
where $q=\rho y_\rho$ and the dot represents derivative with respect to the independent variable found by solving $\dot{\rho}=\rho(1-\rho^2)$. This smooth dynamical system has equilibrium points:
\begin{equation}
(y,q,\rho)\in\big\{(0,0,0),(\pm1,0,0),(\pm1,\tilde{q},1),(0,\tilde{q},1):\tilde{q}\in\mathbb{R}\big\}.
\end{equation}

Our goal is to exclude any solution with $y(0)=0$. We do this by showing that the only solution with $y(0)=0$ is the constant solution $y_*\equiv0$. 
The eigenvalues of $D\dot{Y}(0,0,0)$ are 
\begin{equation}
\Big\{\dfrac{1}{2} (-1 + i\sqrt{11}), \dfrac{1}{2} (-1 - i\sqrt{11}), 1\Big\}
\end{equation}
with corresponding eigenvectors
\begin{equation}
\Bigg\{\begin{bmatrix}\dfrac{1}{6} (-1 - i\sqrt{11})\\1\\0\end{bmatrix}, \begin{bmatrix}\dfrac{1}{6} (-1 + i\sqrt{11})\\ 1\\0\end{bmatrix}, \begin{bmatrix}0\\0\\1\end{bmatrix}\Bigg\}.
\end{equation}
Thus, the unstable manifold of the equilibrium point $(0,0,0)$ is the line defined by $(0,0,\rho)$ for $\rho\in[0,1]$. By the uniqueness of solutions to autonomous dynamical systems, we know that $y=y_*$ is, in fact, the only solution with $y(0)=0$. For if it were not, then any other solution, namely $z$, will have an orbit tangent to $y_*$ only at $\rho=0$. In order for $z$ to not equal $y_*$, the orbit of $z$ must diverge from that of $y_*$. But this cannot be the case since the unstable manifold at $(0,0,0)$ is one-dimensional.

Now, any solution of \eqref{Reduced Self Similar Skyrme Maps ODE} satisfies $y(0)=\pm1$ or is $y_*\equiv0$. We rule out $y_*$ by showing that it is not a minimizer of $J$. First, we notice that $J[y_*]\equiv0$. We can construct a variation of $y_*$ with a smaller value of $J$ by examining the second derivative of $J$ at any $\bar{y}\in X$, $\eta\in C_0^\infty(B_1)$:
\begin{equation}
\frac{d^2}{d\varepsilon^2}J[\bar{y}+\varepsilon\eta]\Big|_{\varepsilon=0}
=\int_0^1\eta_\rho^2\;\;\rho^2d\rho-6\int_0^1(1-3\bar{y}^2)\frac{\eta^2}{\rho^2(1-\rho^2)}\;\;\rho^2d\rho.
\end{equation}
Taking $\bar{y}=y_*$, we get
\begin{equation}
\frac{d^2}{d\varepsilon^2}J[y_*+\varepsilon\eta]\Big|_{\varepsilon=0}
=\int_0^1\eta_\rho^2\;\;\rho^2d\rho-6\int_0^1\frac{\eta^2}{\rho^2(1-\rho^2)}\;\;\rho^2d\rho.
\end{equation}
In \cite{SS_Cauchy}, it is shown that there is an $\bar{\eta}\in C_0^\infty(B_1)$ such that
\begin{equation}
\int_0^1\bar{\eta}_\rho^2\;\;\rho^2d\rho
<6\int_0^1\frac{\bar{\eta}^2}{\rho^2(1-\rho^2)}\;\;\rho^2d\rho.
\end{equation}
Thus,
\begin{equation}
\frac{d^2}{d\varepsilon^2}J[y_*+\varepsilon\bar{\eta}]\Big|_{\varepsilon=0}
< 0.
\end{equation}
This implies that $J[y_*+\varepsilon\bar{\eta}]<J[y_*]$. Since $y_*+\varepsilon\bar{\eta}\in X$, $y_*$ cannot be a minimizer of $J$. Therefore, no minimizer of $J$ will satisfy $y(0)=0$ and, subsequently, it must be the case that $y(0)=\pm1$.
\end{proof}

\begin{remark}
Lemma \ref{y(0) lemma} also proves that a minimizer $y$ of $J$, if it exists, satisifies $J[y]<0$.
\end{remark}
\begin{lemma}
Let $y\in X$ be a minimizer of $J$ with $y(0)=\pm1$ and $y(1)=0$. Then $y$ is monotone.\label{monotone lemma}
\end{lemma}
\begin{proof}
Assume that $y$ is not a monotone function. We will show that $y$ is not a minimizer, contradicting the hypotheses of Lemma \ref{monotone lemma}. Without any loss of generality, we can assume $y(0)=-1$ since anything we show for the other case can be done in the same way. There are two cases to consider: 
\begin{enumerate}
\item $y$ does not exceed $0$ but decreases on some interval and then increases to $0$ (depicted in Figure \ref{Figure 1}), and
\begin{figure}[H]
\centering
\includegraphics[scale=0.8,trim={0 13.5cm 18cm 0},clip]{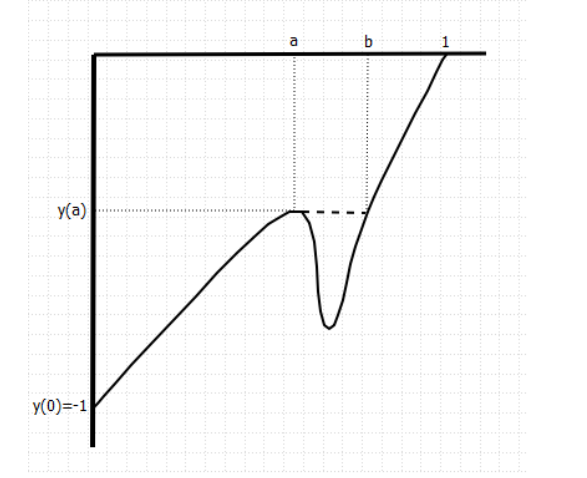}
\caption{Example of case 1. The bold dashed line represents the construction of \eqref{first corrector}.}
\label{Figure 1}
\end{figure}
\item $y$ exceeds $0$ and eventually decreases to $0$ at $\rho=1$ (depicted in Figures \ref{Figure 2} and \ref{Figure 3}).
\begin{figure}[H]
\centering
\includegraphics[scale=0.8,trim={1cm 8cm 17cm 5.3cm},clip]{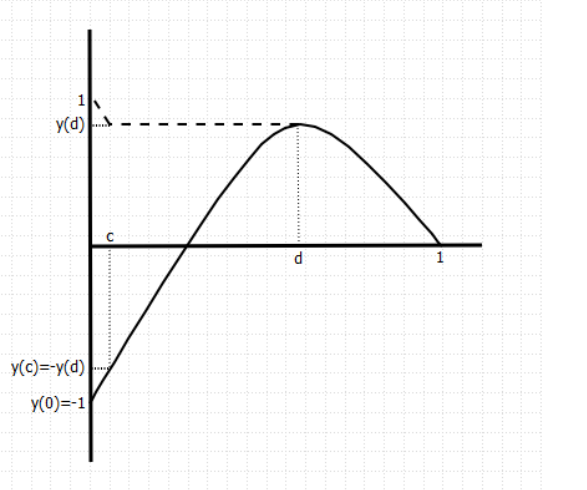}
\caption{Example of case 2, sub-case 1. The bold dashed line represents the construction of \eqref{second corrector}.}
\label{Figure 2}
\end{figure}
\begin{figure}[H]
\centering
\includegraphics[scale=0.8,trim={1cm 9cm 17.9cm 4.9cm},clip]{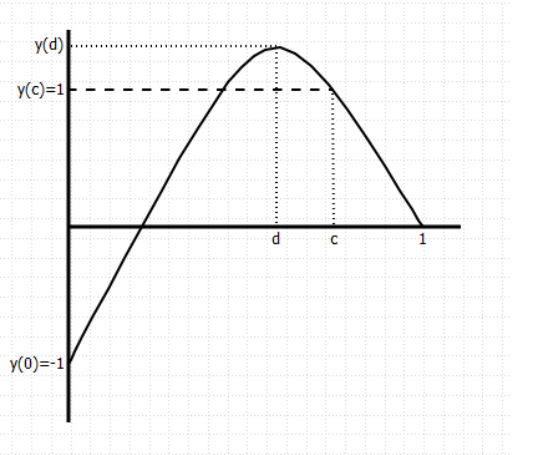}
\caption{Example of case 2, sub-case 2. The bold dashed line represents the construction of \eqref{third corrector}.}
\label{Figure 3}
\end{figure}
\end{enumerate}

In the first case, there exists an interval\footnote{There need not only be one. Where ever such an interval exists, we repeat this process.} $[a,b]$, $0\leq a<b\leq1$, in which $y(a)=y(b)<0$ but $y(a)>y(\rho)$ for $\rho\in(a,b)$. Consider the function
\begin{align}
\tilde{y}(\rho)=
\begin{cases}
y(\rho) &\rho\in[0,a)\bigcup(b,1]
\\
y(a) &\rho\in[a,b].
\end{cases} \label{first corrector}
\end{align}
Since $F[y(a)]<F[y(\rho)]$ for $\rho\in[a,b]$, $J[\tilde{y}]<J[y]$. Thus, we have constructed a new function with a smaller value of $J$.

In the second case, there exists an interval $[a,b]$, $0<a<b\leq1$ in which $y(a)=y(b)=0$ but $y(\rho)>0$ for $\rho\in(a,b)$. Further, there exists $d\in(a,b)$ such that $y(\rho)\leq y(d)$ for all $\rho\in[a,b]$. There are now two sub-cases to consider: $y(d)<1$ and $y(d)\geq1$.

If $y(d)<1$, then there must be some $c<d$ such that $y(c)=-y(d)$. We want to reflect the portion of the graph of $y$ before $c$ and then repeat the process used in the first case. This is done by considering the function
\begin{align}
\tilde{y}(\rho)=
\begin{cases}
-y(\rho) &\rho\in[0,c)
\\
y(d) &\rho\in[c,d]
\\
y(\rho) &\rho\in(d,1].
\end{cases} \label{second corrector}
\end{align}
Since $F[y(d)]<F[y(\rho)]$ for $\rho\in[c,d]$, $J[\tilde{y}]<J[y]$. 

If $y(d)\geq1$, then there must be some $c>d$ such that $y(c)=1$. We then consider the function
\begin{align}
\tilde{y}(\rho)=
\begin{cases}
1 &\rho\in[0,c]
\\
y(\rho) &\rho\in(c,1].
\end{cases} \label{third corrector}
\end{align}
Since $F[1]<F[y(\rho)]$ for $\rho\in[0,c]$, $J[\tilde{y}]<J[y]$.

In each case, we have shown that a non-monotone minimizer of $J$ with $y(0)=\pm1$ and $y(1)=0$ is not actually a minimizer of $J$. Therefore, a minimizer of $J$, $y\in X$, with $y(0)=\pm1$ and $y(1)=0$ is a monotone function.
\end{proof}

\begin{lemma}
$J$ attains its minimum in $X$ at a smooth function $y$ such that $-1\leq y\leq1$. \label{Minimum Attained and Smoothness Achieved}
\end{lemma}
\begin{proof}
We employ an argument similar to that of the proof of the existence of a minimizer for an energy functional used in \cite{Analysis_LL}, page 276. Let $\{y_n\}$ be a minimizing sequence of $J$. That is, $\lim_{n\rightarrow\infty}J[y_n]=\inf_{\psi\in X}J[\psi]:=J_0$. By \eqref{Functional is Bounded Below}, $\{y_n\}$ is a bounded sequence in $H^1$. The Banach-Alaglou Theorem implies that there is a subsequence, also denoted $\{y_n\}$ which is weakly convergent in $H^1$ and strongly convergent in $L^2$ to a function $y\in X$. Furthermore, there exists a constant $C$ such that
\begin{equation}
\frac{y_n^2\Big(1-\frac{1}{2}y_n^2\Big)}{\rho^2(1-\rho^2)}
\leq C\Big(1+\frac{1}{\rho^2}\Big)
\end{equation}
for all $n$. This is certainly integrable on $B_1$. Even further, $y_n^2\Big(1-\frac{1}{2}y_n^2\Big)\rightarrow y^2\Big(1-\frac{1}{2}y^2\Big)$ almost everywhere. By the Dominated Convergence Theorem and weak lower semicontinuity of the $H^1$ norm,
\begin{align}
\begin{split}
J[y]&=J\big[\lim_{n\rightarrow\infty}y_n\big]
\\
&\leq\lim_{n\rightarrow\infty}J[y_n] \label{minimum}
\\
&=:J_0.
\end{split}
\end{align}
Since $J_0$ is the infimum of $J$, \eqref{minimum} implies $J[y]=J_0$. Consequently, the convergence is strong in $H^1$. Therefore, $J$ attains its minimum at a function $y\in X$. Further, $y$ is continuous since the $\alpha$-limit set of the corresponding smooth dynamical system in Lemma \ref{y(0) lemma} tells us $y(0)=\pm1$. By Lemma \ref{monotone lemma}, $y$ is also monotone. Thus, $|y|\leq1$. Therefore, by Lemma \ref{Smoothness Lemma}, $y$ is a smooth function of $\rho\in[0,1]$.
\end{proof}

\begin{remark}
Since the solution satisfies $|y|\leq1$, our minimization problem is independent of the regularization we placed on \eqref{Reduced Self Similar Skyrme Maps Minimization Functional}. Thus, Lemmas \ref{Smoothness Lemma}-\ref{Minimum Attained and Smoothness Achieved} are true for \eqref{Reduced Self Similar Skyrme Maps Minimization Functional} as well as \eqref{Regularized Reduced Self Similar Skyrme Maps Minimization Functional}.
\end{remark}

Now we can state the proof of Theorem \ref{Main Result}:
\begin{proof}
As previously stated, the Euler-Lagrange equation of the strong-field, equivariant Skyrme Map $U(t,r,\omega)=\big(u(t,r),\omega\big)$ is given by
\begin{equation}
u_{tt}-u_{rr}-\frac{2}{r}u_r+\cos u\bigg(\frac{u_t^2-u_r^2}{\sin u}+\frac{3\sin u}{r^2}\bigg)=0. \label{Full Skyrme Maps PDE}
\end{equation}
Let $\phi(x)$, $x=(r,\omega)\in\mathbb{R}^{5}$ and $\omega\in S^4\subset\mathbb{R}^5$, be the smooth function defined by
\begin{equation}
\phi(x)=(\cos^{-1} y(r),\omega)
\end{equation}
where $y$ is a smooth solution to \eqref{Reduced Self Similar Skyrme Maps ODE} with $y(0)=\pm1$ and $y(1)=0$. We can supply the following Cauchy data to \eqref{Full Skyrme Maps PDE}:
\begin{align}
\begin{split}
&U(-1,x)=\phi(x)
\\
\partial_t&U(-1,x)=x^i\partial_i\phi \label{Smooth Cauchy Data}
\end{split}
\end{align}
Here, it is implicitly understood that the equations for the angles, $\omega$, is trivial. Then in the past light cone of the origin of the Minkowski space, the solution is 
\begin{equation}
U(t,x)=\phi(-x/t).
\end{equation}
Since the solution is multivalued at the origin, $\partial_iU(t,0)\rightarrow\infty$ as $t\rightarrow0$.
\end{proof}

This concludes the argument and shows that there is a class of smooth initial data for the equivariant, strong-field Skyrme Model equation of motion which develops a singularity in finite time.

\section*{Acknowledgements}
Michael McNulty would like to thank Professor Shadi Tahvildar-Zadeh for suggesting this problem and for many illuminating conversations. 

\bibliographystyle{plain}
\bibliography{References}

\end{document}